\documentclass[11pt]{article}
\usepackage{latexsym,amsmath,amssymb,amsfonts,mathrsfs,amsthm}
\usepackage{epsf,graphicx,epsfig,color,cite,cases}
\usepackage{subfigure,graphics,multirow,marginnote,enumerate,bm}
\usepackage{algorithm}

\newcommand{\R}{{\mat R}}

\newcommand{\C}{{\mat C}}

\newcommand{\ds}{\displaystyle}
\newcommand{\no}{\nonumber}
\newcommand{\be}{\begin{eqnarray}}
\newcommand{\ben}{\begin{eqnarray*}}
\newcommand{\en}{\end{eqnarray}}
\newcommand{\enn}{\end{eqnarray*}}

\newcommand{\pa}{\partial}

\newcommand{\ov}{\overline}

\newcommand{\G}{\Gamma}

\newcommand{\Om}{\Omega}

\newcommand{\al}{\alpha}

\newcommand{\wi}{\widetilde}

\newcommand{\mat}{\mathbb}
\newcommand{\se}{\setminus}

\newcommand{\la}{\lambda}

\newtheorem{theorem}{Theorem}[section]
\newtheorem{lemma}[theorem]{Lemma}

\newtheorem{remark}[theorem]{Remark}

\begin{document}
\renewcommand{\theequation}{\arabic{section}.\arabic{equation}}

\title{\bf A linear sampling method for inverse acoustic scattering by a locally rough interface
}

\author{Jianliang Li\thanks{School of Mathematics and Statistics, Changsha University of Science
and Technology, Changsha 410114, P.R. China ({\tt lijl@amss.ac.cn})}
\and
Jiaqing Yang\thanks{School of Mathematics and Statistics, Xi'an Jiaotong University,
Xi'an, Shaanxi 710049, China ({\tt jiaq.yang@mail.xjtu.edu.cn})}
\and
Bo Zhang\thanks{LSEC, NCMIS and Academy of Mathematics and Systems Science, Chinese Academy
of Sciences, Beijing 100190, China and School of Mathematical Sciences, University of Chinese
Academy of Sciences, Beijing 100049, China ({\tt b.zhang@amt.ac.cn})}
}
\date{}

\maketitle

\begin{abstract}
This paper is concerned with the inverse problem of time-harmonic acoustic scattering by an unbounded,
locally rough interface which is assumed to be a local perturbation of a plane.
The purpose of this paper is to recover the local perturbation of the interface from the near-field
measurement given on a straight line segment with a finite distance above the interface and
generated by point sources. Precisely, we propose a novel version of the linear sampling method
to recover the location and shape of the local perturbation of the interface numerically.
Our method is based on a modified near-field operator equation associated with a special rough surface,
constructed by reformulating the forward scattering problem into an equivalent integral equation
formulation in a bounded domain, leading to a fast imaging algorithm.
Numerical experiments are presented to illustrate the effectiveness of the imaging method.

\vspace{.2in}


\end{abstract}

\maketitle

\section{Introduction}

Consider the inverse scattering problem of determining an unbounded rough surface in
a two-layered dielectric media from near-field measurements. This kind of problems plays a
fundamental role in diverse scientific areas such as radar and sonar detection, underwater
exploration and non-destructive testing. In such applications, one tries to recover the shape and
location of unknown surfaces by taking the scattered wave-field measurements in certain domain.

Precisely, we consider the scattering of a time-harmonic line source
by an unbounded rough surface. The surface is assumed to
be a local perturbation of a plane which separates the whole space into two unbounded parts.
We are restricted to the two-dimensional case by assuming that the interface and the line source
are invariant in the $x_3$ direction.
Then the wave motion is governed by the two-dimensional Helmholtz equation with the wavenumber
described by a piecewise constant function. Due to the assumption of the interface,
the Sommerfeld radiation condition remains valid to describe the behavior of the scattered wave
away from the rough interface. Given the incident wave (which is a point source in 2D) and the rough
interface, the forward problem is to find the distribution of the scattered wave-field in the whole space.
Many works have been done to study the existence of unique solutions to the forward scattering problem.
We refer the reader to \cite{SE10,MT06,LYZ13,DTS03,QZZ19,ZS03,ZZ13} with either the variational technique
or the boundary integral equation method.

We are mainly interested in the inverse problem of recovering the shape and location of the unknown
interface from the scattered near-field measurements. In \cite{YLZ17}, a global uniqueness theorem
has been proved in a more general case for simultaneously determining the locally rough interface together
with the wavenumber and embedded obstacles in the lower half-space from the near-field measurements
in the upper-half space.
Following this uniqueness result, in this paper we aim to develop an efficient sampling-type method to
solve the associated inverse problem of recovering the unbounded rough interface in the case with no 
embedded obstacles. 
For the case of unbounded impenetrable surfaces, the above inverse problem has
been extensively studied numerically, most of which focused on iterative optimization algorithms
under some a priori information on the surface. We refer to \cite{CR10} for a Kirsch-Kress approach,
\cite{LB13} for a nonlinear integral equation method, \cite{BGL11,GJ11,GJ13,CP02,QZZ19,RM15,ZZ13} for
Newton-type iterative algorithms, and \cite{GL13} for an inversion algorithm based on the transformed 
field expansion in the case when the rough surface is a small and smooth perturbation of a plane. 
In addition, several non-iterative sample methods have also been studied
for locating the unbounded rough surface, such as a time domain point source method \cite{LC05},
the factorization method \cite{AL08} and the linear sampling method \cite{DLLY17} for the case
with a Dirichlet condition on the surface.
It is noted that the inversion algorithm in \cite{AL08} was shown to be valid
theoretically only for $\kappa f_+\in (0,\sqrt{2})$, where $\kappa>0$ stands for the wavenumber
and $f_+$ denotes the amplitude of the rough surface. It remains open to show the effectiveness of
the factorization method for the unbounded surface in a general case for $\kappa\in\R_+$.
If the surface is considered to be a penetrable interface in dielectric media, few inversion algorithms
are available for recovering the shape and location of the surface. 
A Kirsch-Kress method was proposed in \cite{LZ13} for reconstructing a locally perturbed plane interface, 
making use of near-field measurements both above and below the interface. An Newton-type iteration algorithm
was introduced in \cite{CG11} to reconstruct a locally rough interface from far-field measurements. 
Further, the inversion algorithm proposed in \cite{GL13} was extended in \cite{GL14} to deal with the case 
where the rough surface is a small and smooth perturbation of a plane interface in dielectric media.
Recently, a direct imaging method has been proposed to reconstruct an unbounded rigid 
rough surface either from the elastic scattered near-field Cauchy data generated by point sources in \cite{LZZ19} 
or from the elastic scattered near-field data generated by elastic plane waves in \cite{HLZZ19} and
to recover an impenetrable or penetrable unbounded rough surface either from the acoustic scattered 
near-field Cauchy data generated by point sources in \cite{LZZ18} or from the acoustic scattered near-field 
data generated by plane waves in \cite{Zhang20}.
Further, a direct imaging algorithm was developed in \cite{XZZ19} for recovering an unbounded Dirichlet rough 
surface from phaseless near-field data generated by plane waves. 
However, it seems difficult to develop a linear sampling method for locating an unbounded interface in 
dielectric media. As far as we know, no such a result is available so far.

In this paper we will investigate the linear sampling method (LSM) as an analytical as well as a numerical
tool to solve the inverse scattering problem of recovering a locally rough interface in dielectric media
from near-field measurements in a certain domain. It is well-known that the classical LSM was first
introduced \cite{DA96} for inverse acoustic scattering by bounded obstacles and has been extended to
many other inverse problems since then (see, e.g., \cite{CK13}) since the reduced algorithm is fast and
does not need any a priori knowledge on the scatterers.
We remark that a modified version of the LSM was recently proposed in \cite{DLLY17}, where an auxiliary
rough surface is introduced in order to recover a locally rough surface with a Dirichlet boundary condition.
Partially motivated by \cite{DLLY17} and \cite{YLZ17}, we reformulate the interface scattering problem into
an equivalent integral equation formulation in a bounded domain by introducing a class of special rough
surfaces, where the well-posedness easily follows from the classical Riesz-Fredholm alternative.
With this technique, a modified near-field operator equation associated with a special surface
will be constructed, which is proved in a strict way to be valid for recovering the shape and location of
the locally rough interface, leading to a fast imaging algorithm. As two related works, we also
refer the reader to \cite{YZZ1,YZZ2} for a periodical version.

The remaining part of the paper is organized as follows. In Section \ref{sec2}, we introduce the
mathematical model for the forward scattering problem.
In Section \ref{sec3}, we propose a novel version of the classical LSM by constructing a modified
near-field operator equation associated with a special rough surface.
In Section \ref{sec4}, numerical experiments are conducted to demonstrate the effectiveness
of the proposed method.

\section{The mathematical model}\label{sec2}
\setcounter{equation}{0}

We now introduce the mathematical formulation of the scattering problem by an unbounded
rough interface in two dimensions. The unbounded interface is denoted by
the curve $\G=\{(x_1,x_2)\in\R^2: x_2=f(x_1), x_1\in\R^1\}$, where $f\in C^{2,\alpha}(\R^1)$,
$\alpha\in (0,1)$, is a smooth function (see Figure \ref{f1}), where
the function $f$ is assumed to satisfy the condition that
there exists a positive constant $M>0$ such that
\be\label{Condition-A}
f(x_1)=0\qquad {\rm for}\;\;\;|x_1|\geq M.
\en
This condition means that the interface $\G$ is only different from the plane
$\G_0=\{(x_1,x_2)\in\R^2: x_2=0\}$ for $x_1$ in a finite interval $(-M,M)$. \
Let $\Om_1:=\{(x_1,x_2)\in\R^2: x_2>f(x_1), x_1\in\R^1\}$ denote
the unbounded domain above $\Gamma$ which is filled with a homogeneous material characterized
by a constant wavenumber $\kappa_1>0$. Denote by $\Om_2:=\R^2\se\ov{\Om_1}$ the complement of $\Om_1$
in $\R^2$ which is filled with another homogeneous material characterized by a different
wavenumber $\kappa_2>0$.

\begin{figure}[htbp]
\centering
\includegraphics[scale=0.5]{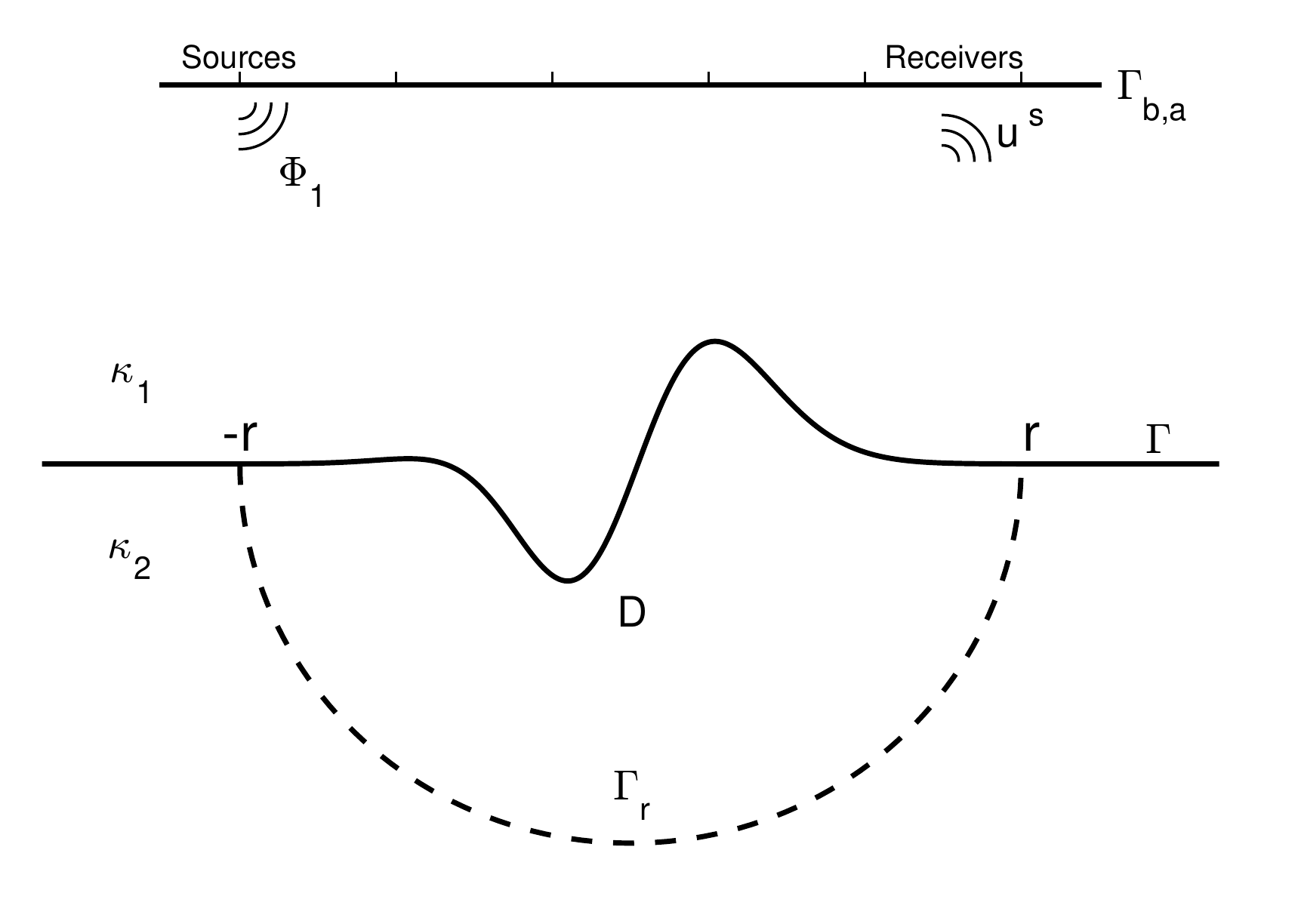}
\caption{The physical configuration of the scattering problem}\label{f1}
\end{figure}

Consider an incoming wave induced by the point source in 2D
\ben
u^{\rm i}(x) = \Phi_{\kappa_1}(x,y):=\frac{\rm i}{4}H_0^{(1)}(\kappa_1|x-y|)\qquad
{\rm for\;}\;\;y\in\Om_1,
\enn
where $H_0^{(1)}(\cdot)$ is the Hunkel function of the first kind of order zero and
$\Phi_{\kappa_1}(\cdot,\cdot)$ denotes the fundamental solution of the two-dimensional Helmholtz equation
satisfying $\Delta \Phi_{\kappa_1}+\kappa_1^2\Phi_{\kappa_1}=-\delta_y$ in $\R^2$
with $\delta_y(\cdot)=\delta(\cdot-y)$. The associated wavelength of the incident wave is then given
by $\la=2\pi/\kappa_1$. Thus the scattering problem of $u^{\rm i}$ by $\G$ is modelled
by the Helmholtz equation
\be\label{a1}
\Delta u+\kappa^2(x)u=-\delta_y\qquad{\rm in}\;\;\; \R^2,
\en
where $u:=u^{\rm i}+u^{\rm s}$ denotes the total field in $\Omega_1$ consisting of the incident field
$u^{\rm i}$ and the scattered field $u^{\rm s}$, and $u:=u^{\rm s}$ in $\Omega_2$ denotes
the transmitted field. Moreover, $\kappa(x)$ is defined as $\kappa(x):=\kappa_1$ for $x\in\Omega_1$
and $\kappa(x):=\kappa_2$ for $x\in\Omega_2$.

Due to the condition (\ref{Condition-A}) on the surface $\G$, the Sommerfeld radiation condition
\be\label{a2}
\ds\lim_{|x|\rightarrow \infty}\sqrt{|x|}\left(\frac{\partial u^{\rm s}}{\partial |x|}
-{\rm i}\kappa u^{\rm s}\right)=0
\en
remains valid to describe the asymptotic behavior of the scattered field $u^{\rm s}$ away from $\G$.
It is remarked that the Sommerfeld radiation condition (\ref{a2}) should be replaced by the much weaker
Upward and Downward Propagating Radiation Conditions (cf. \cite{SZ99}) for the globally rough interface case,
which can be shown to be equivalent to (\ref{a2}) for the locally rough interface case. The reader is referred
to \cite{DLLY17,LZ13} for more detailed discussions.

Uniqueness of solutions to the scattering problem (\ref{a1})-(\ref{a2}) can be found in \cite{SZ99}
(see Theorem 5.1 of \cite{SZ99}), where a more complicated scattering problem was dealt with for an infinite
inhomogeneous conducting or dielectric layer at the interface.
The existence of solutions to the problem (\ref{a1})-(\ref{a2}) can be established by the integral equation
method using the Green function of the two-layered medium, that is, the fundamental solution of the
unperturbed problem ($f=0$; cf. \cite{Li2010}).
It can also be proved by either using the boundary integral equation method in \cite{DTS03} or
reducing (\ref{a1})-(\ref{a2}) into an equivalent Lippmann-Schwinger integral equation in a bounded domain
based on the Green function $G_r$ below (see \cite{YLZ17}).

To propose our sampling method in this paper, we follow \cite{YLZ17} to introduce a special class of rough
interfaces $\G_r$ (see Figure \ref{f1}) defined by
\ben
\G_r:=\{(x_1,x_2)\in\R^2: x_2=0\quad{\rm for}\;\;|x_1|\geq r,\;\;{\rm and}\;\; x_2=-\sqrt{r^2-x_1^2}\quad{\rm for}\;\;|x_1|< r\}
\enn
for each fixed $r>0$. It can be seen that $\G_r$ differs from the exact scattering interface $\G$ only in
a finite interval. Let $\Omega_{1,r}$ and $\Omega_{2,r}$ denote the unbounded domains above and below
the interface $\G_r$, respectively. Throughout this paper, we choose a sufficiently large $r>M$ such that
the local perturbation of $\G$ lies totally inside the region $\Omega_{1,r}$, that is,
$\Omega_1\subset\Omega_{1,r}$.

Consider the scattering of the incident point source $\Phi_{\kappa_1}(\cdot, y)$ for
$y\in\Omega_{1,r}$ by $\G_r$, which is to find the solution $G_r(\cdot,y)$ satisfying the problem
\be\label{c1}
\left\{\begin{array}{ll}
\Delta G_r(x,y)+\kappa_r^2(x)G_r(x,y)=-\delta_y& {\rm in}\;\;\R^2\\[2mm]
\ds\lim_{|x|\rightarrow\infty}\sqrt{|x|}\left(\frac{\pa G_r(x,y)}{\pa |x|}-{\rm i}\kappa_r G_r(x,y)\right)=0&
\end{array}\right.
\en
where the wavenumber $\kappa_r(x)$ is defined as $\kappa_r(x):=\kappa_j$ for $x\in\Omega_{j,r}$, $j=1,2$. It follows from \cite{YLZ17} that Problem (\ref{c1}) is uniquely
solvable for each fixed $y\in\Om_{1,r}$.
It is also well-known that $G_r(\cdot,\cdot)$ corresponds to the
Green's function associated with the special interface $\G_r$. Similar to (\ref{a1})-(\ref{a2}),
$G_r(\cdot,y)$ can be decomposed into the sum of the incident field $\Phi_{\kappa_1}(\cdot, y)$
and its scattered field $G_r^{\rm s}(\cdot,y)$ in $\Om_{1,r}$ for $y\in\Om_{1,r}$, i.e.,
$G_r(\cdot,y)=\Phi_{\kappa_1}(\cdot,y)+G_r^{\rm s}(\cdot,y)$ in $\Om_{1,r}$.

Next, we shall briefly derive an equivalent integral equation of Problem (\ref{a1})-(\ref{a2})
with the help of Problem (\ref{c1}), following the idea in \cite{YLZ17}. To this end,
denote by $D=\Om_2\cap\Om_{1,r}$ the intersection between $\G_r$ and $\G$. Applying
the Green's formula for the difference $v_r(x; y):=u(x,y)-G_r(x,y)$, it can be verified that
the solution $u(\cdot,y)$ to Problem (\ref{a1})-(\ref{a2}) satisfies the Lippmann-Swinger integral
equation in $D$:
\be\label{c3}
u(\cdot,y)+\eta\int_{D}G_r(\cdot,z)u(z,y){\rm d}z=G_r(\cdot,y)\qquad {\rm for\;\;}
\eta:=\kappa_1^2-\kappa_2^2.
\en
Define the map $T: L^2(D)\rightarrow L^2(D)$ by
\ben
(T\phi)(x):=\int_{D}G_r(x,z)\phi(z){\rm d}z\qquad{\rm for}\;\; x\in D.
\enn
Then the equation (\ref{c3}) can be rewritten in the operator equation form:
\be\label{c4}
(I+\eta T)u(\cdot,y)=G_r(\cdot,y)
\en

By a similar discussion as in Theorem 8.3 of \cite{CK13}, we have the following equivalence result
between the solvability of Problems (\ref{a1})-(\ref{a2}) and (\ref{c4}).

\begin{theorem}\label{thm2.1}
For $y\in\Om_1$, if $u(\cdot,y)$ is a solution of (\ref{a1})-(\ref{a2}), then $u(\cdot,y)|_{D}$ is a solution of (\ref{c4}). Conversely, if $u(\cdot,y)$ is a solution of (\ref{c4}), then
\ben
u(x,y):=G_r(x,y)-\eta\int_{D}G_r(x,z)u(z,y){\rm d}z\qquad {\rm for}\;\; x\in\R^2,
\enn
is a solution of (\ref{a1})-(\ref{a2}). Furthermore, the map $I+\eta T$ is an isomorphism on $L^2(D)$.
\end{theorem}

\section{The inverse scattering problem}\label{sec3}
\setcounter{equation}{0}

Based on Theorem \ref{thm2.1}, we study in this section the LSM as an analytical as well as a numerical tool to solve the inverse scattering problem of recovering the shape and location of
the unbounded interface $\G$ by the knowledge of scattered fields $u^{\rm s}(\cdot,y)$ on the
measurement segment $\G_{b,a}:=\{(x_1,x_2)\in\R^2: x_2=b, |x_1|<a\}$. Here, $a>0$ and
$b\in\R$ such that $b>f_+:=\sup_{x_1\in\R}f(x_1)$. Similar to the bounded case \cite{DA96},
our objective is to define an indicator function $I(z)$ by the $L^2$ norm of the solution to
the first kind of integral equation associated with the near-field operator
${\bf N}:L^2(\Gamma_{b,a})\rightarrow L^2(\Gamma_{b,a})$:
\be\label{d1}
({\bf N}g)(x): = \int_{\G_{b,a}}u^{\rm s}(x,y)g(y){\rm d}s(y)\qquad \textrm{for}\;\; x\in\G_{b,a},
\en
or its modified version, which can provide a fast imaging algorithm for $\G$. In (\ref{d1}) $u^{\rm s}(\cdot,y)$ indicates the scattered solution to Problem (\ref{a1})-(\ref{a2}) for the incident point
source $\Phi_{\kappa_1}(\cdot,y)$.

Recalling for the bounded obstacle case, the far-field operator $F$ could be decomposed into
$F=\Lambda H$ with $\Lambda$ a solution operator and $H$ an incidence operator for the obstacle scattering problem. However, it becomes very difficult to directly obtain such decomposition in the unbounded rough surface case especially for a penetrable interface.
To overcome this difficulty, we instead consider a modified near-field operator
 ${\bf N}_{\rm Mod}:L^2(\G_{b,a})\rightarrow L^2(\G_{b,a})$:
 \be\label{d22}
({\bf N}_{\rm Mod}g)(x):=\int_{\G_{b,a}}(u^{\rm s}(x,y)-G_r^{\rm s}(x,y))g(y){\rm d}s(y)\qquad \textrm{for}\;\; x\in\G_{b,a}
\en
by the difference $v_r(x,y):=u(x,y)-G_r(x,y)$ which is easily checked from (\ref{a1})-(\ref{a2}) and (\ref{c1}) to satisfy the boundary value problem
\be\label{c2}
\left\{\begin{array}{llllll}
\Delta v_r(\cdot,y)+\kappa_1^2v_r(\cdot,y)=0 & {\rm in}\;\; \Om_1\\[2mm]
\Delta v_r(\cdot,y)+\kappa_2^2v_r(\cdot,y)=\beta(\cdot) &{\rm in}\;\; D\\[2mm]
\Delta v_r(\cdot,y)+\kappa_2^2v_r(\cdot,y)=0 &{\rm in}\;\; \Om_{2,r}\\[2mm]
[v_r(\cdot,y)] =[\pa_{n} v_r(\cdot,y)]=0& {\rm on}\;\; \G\cup\G_r\\[2mm]
\ds\lim_{|x|\rightarrow \infty}\sqrt{|x|}\left(\frac{\pa v_r(\cdot,y)}{\pa |x|}-{\rm i}\kappa v_r(\cdot,y)\right)=0&\\[2mm]
\end{array}\right.
\en
where $\beta(\cdot):=\eta G_r(\cdot,y)$, and $[\cdot]:=\cdot|_{+}-\cdot|_{-}$ with $\cdot|_{\pm}$ indicating  the limits of $\cdot$ from the upward and downward of $\G$ or $\G_{r}$, respectively.
The well-posedness of Problem (\ref{c2}) easily follows from the well-posedness of the original scattering problem (\ref{a1})-(\ref{a2}).
Notice that Problem (\ref{c2}) may be viewed as the scattering by a bounded inhomogeneous
medium $D$ with its boundary containing the local perturbation of the interface $\G$.
Therefore, our sampling method will be based on studying the following integral equation
\be\label{d2}
({\bf N}_{\rm Mod}g)(x)=G_r(x,z)|_{\G_{b,a}}\qquad {\rm for}\;\;z\in\R^2.
\en
It is expected to define an indicator function $I_{\rm Mod}(z)$ by the $L^2$ norm of the density $g$ in  (\ref{d2}), which asymptotic behavior could be used to recover the shape and location of $D$ or $\G$.

To illustrate this point, we introduce the function
\be\label{d3}
v_g(x): = \int_{\G_{b,a}}G_r(x,y)g(y){\rm d}s(y)\qquad x\in D,
\en
for $g\in L^2(\G_{b,a})$ which clearly satisfies the Helmholtz equation $\Delta v_g+\kappa_1^2 v_g=0$ in $D$. Define the space $X:=\{v_g|_{D}: g\in L^2(D)\}$ and
denote by $\ov{X}$ the closure of $X$ in the sense of $L^2(D)$ norm. Moreover,
we introduce the operator $F: L^2(D)\rightarrow L^2(\G_{b,a})$ defined by
\ben
(Fh)(x): = -\eta\int_{D}G_r(x,z)h(z){\rm d}z\qquad{\rm for}\;\; x\in\G_{b,a}.
\enn
By combining equation (\ref{c4}) and Theorem \ref{thm2.1}, we have the following relation
\be\no
({\bf N}_{\rm Mod}g)(x)&=&\int_{\Gamma_{b,a}}(Fu(\cdot,y))(x)g(y){\rm d}s(y)\\\label{d4}
&=&\int_{\Gamma_{b,a}}F(I+\eta T)^{-1}G_r(x,y)g(y){\rm d}s(y)\\\no
&=&F(I+\eta T)^{-1}v_g(x)
\en
from the definition (\ref{d22}). Thus, the near-field operator equation (\ref{d2}) can be
reformulated as
\be\label{d5}
F(I+\eta T)^{-1}v_g(\cdot)=G_r(\cdot,z)|_{\G_{b,a}}\qquad {\rm for} \;\;z\in\R^2.
\en

In order to investigate the solvability of equation (\ref{d5}), similar to the bounded case, we introduce the interior transmission problem in finding a pair of functions $(v,w)\in L^2(D)\times L^2(D)$ such that  $(v,w)$ satisfies the Helmholtz equations
\be\label{d6}
\Delta v+\kappa_1^2v=0,\quad \Delta w+\kappa_1^2q w=0\qquad{\rm in}\;\; D,
\en
and the transmission conditions
\be\label{d7}
w-v=G_r(\cdot,z)\quad \frac{\pa w}{\pa n}-\frac{\pa v}{\pa n}=\frac{\pa G_r(\cdot,z)}{\pa n}
\qquad{\rm on }\;\;\pa D,
\en
for $z\in D$ and $q:=\kappa_2^2/\kappa_1^2$. It is well-known that $\kappa_1^2>0$ is
called a transmission eigenvalue if the homogeneous problem (\ref{d6})-(\ref{d7}) has a pair
of nontrivial solutions $(v,w)\in L^2(D)\times L^2(D)$ such that $v-w\in H_0^2(D)$.
For simplicity, we always assume throughout the paper that {\em $\kappa_1^2>0$ is not a transmission eigenvalue}. Under this assumption, we conclude from \cite{CGH10} that Problem (\ref{d6})-(\ref{d7}) has a unique solution $(v(\cdot,z),w(\cdot,z))\in L^2(D)\times L^2(D)$ for
$z\in D$.

Furthermore, we introduce the space $Y$ defined by
\ben
Y:=\{u\in L^2(D): \Delta u\in L^2(D),\;\Delta u+\kappa_1^2 u=0  \;{\rm in}\; D\;{\rm in\;the\;distributional\;sense}\}
\enn
that is closely related to the set $X$.  It is easily seen that both $\ov{X}$ and $Y$ are closed subspaces of the Hilbert space $L^2_{\Delta}(D):=\{u\in L^2(D): \Delta u\in L^2(D)\}$ equipped with the norm induced by the inner product
$(u,v)_{L^2_{\Delta}(D)}:=(u,v)_{L^2}+(\Delta u,\Delta v)_{L^2}$ for $u,v\in L^2_{\Delta}(D)$.
Hence, both $\ov{X}$ and $Y$ are also Hilbert spaces in the sense of the above norm. For further analysis in next subsections, we need to show that Problem
(\ref{d6})-(\ref{d7}) has a unique solution $(v(\cdot,z),w(\cdot,z))\in \ov{X}\times L^2(D)$ for each
$z\in D$. Thus, it is enough to show the coincidence between $\ov{X}$ and $Y$.
\begin{lemma}\label{lem}
If $\kappa_1^2>0$ is not a Dirichlet eigenvalue of $-\Delta$ in $D$, then $\ov{X}=Y$.
\end{lemma}

\begin{proof}
To prove the assertion, we shall first show that the space $Y$ can be described by the
solution of the Dirichlet problem
\be\label{d13}
\left\{\begin{array}{ll}
\Delta \hat{w}+\kappa_1^2\hat{w}=0 &\quad {\rm in}\;\; D,\\[2mm]
\hat{w}=\hat{g}\in H^{-\frac{1}{2}}(\pa D) &\quad {\rm on}\;\; \pa D.
\end{array}\right.
\en
Hence, we need to prove the uniquely solvability of Problem (\ref{d13}) in $L^2_\Delta(D)$ which
will be done by employing boundary integral equation techniques with seeking a solution
in the form of a combined double and single-layer potential
\be\label{d14}
\hat{w}(x) = \int_{\pa D}\left(\frac{\pa \Phi_{\kappa_1}(x,y)}{\pa n(y)}-{\rm i} \Phi_{\kappa_1}(x,y)\right)\psi(y){\rm d}s(y) \qquad {\rm{for}}\;\; x\in \R^2\se \pa D,
\en
where the density $\psi\in H^{-\frac{1}{2}}(\pa{D})$ and $n(y)$ stands for the normal derivative at $y\in\pa D$ directed into the exterior of $D$.
Then we deduce by the boundary condition in (\ref{d13}) that the potential $\hat{w}$ in (\ref{d14}) solve Problem (\ref{d13}), provided the density $\psi$ is a solution of a second-kind of integral equation
\be\label{d15}
(I+K-{\rm i}S)\psi = 2\hat{g}
\en
with the single- and double-layer operators $S$ and $K$, respectively, given by
\ben
(S\psi)(x): &=&  2\int_{\partial D}\Phi_{\kappa_1}(x,y)\psi(y)ds(y)\qquad\quad x\in\partial D,\\
(K\psi)(x): &=&  2\int_{\partial D}\frac{\pa \Phi_{\kappa_1}(x,y)}{\pa n(y)}\psi(y)ds(y)\qquad x\in\partial D.
\enn
It follows from Corollary 3.7 of \cite{CK13} that both $S$ and $K$ are bounded operators from
$H^{-\frac{1}{2}}(\pa D)$ into $H^{\frac{1}{2}}(\pa D)$ due to the piecewise $C^{2,\al}$-regularity
of $\pa D$. This, together with the compact imbedding of $H^{\frac{1}{2}}(\pa D)$ into $H^{-\frac{1}{2}}(\pa D)$,
implies that  $I+K-{\rm i}S: H^{-\frac{1}{2}}(\pa D)\rightarrow H^{-\frac{1}{2}}(\pa D)$ is of Fredholm type
with index $0$. The existence of a solution of equation (\ref{d15}) then follows from the Riesz-Fredholm
theory and the uniqueness of Problem (\ref{d13}) for each $\hat{g}\in H^{-1/2}(\pa D)$.
Hence, the space
$\{\hat{w}\in L^2_{\Delta}(D):\hat{w}{\rm\;solves\;Problem\;(\ref{d13})\;for\;}\hat{g}\in H^{-1/2}(\pa D)\}$
coincides with $Y$.

Now, by the trace theorem it is sufficient to prove that $v_g|_{\pa D}$ is dense in $H^{-1/2}(\pa D)$,
which is equivalent to show that the range ${\rm Range}(H)$ of $H: L^2(\G_{b,a})\rightarrow H^{-1/2}(\pa D)$
defined by
\ben
(Hg)(x):=\int_{\G_{b,a}}G_r(x,y)g(y){\rm d}s(y)\qquad x\in\pa D
\enn
is dense in $H^{-1/2}(\pa D)$.
By interchanging the order of integration, the dual operator $H^*:H^{1/2}(\pa D)\rightarrow L^2(\G_{b,a})$
of $H$ is given by
\ben
(H^*\psi)(x)&:=& \int_{\pa D}\ov{G_r(y,x)}\psi(y){\rm d}s(y)\\
            &=& \int_{\pa D}\ov{G_r(x,y)}\psi(y){\rm d}s(y)\qquad x\in\G_{b,a}.
\enn
Here, we have used the reciprocity relation $G_r(x,y)=G_r(y,x)$ which holds for all $x\in\G_{b,a}$
and $y\in\pa D$. Notice that the boundary $\pa D$ is composed of $\G\se\G_r$ and $\G_r\se \G$.
For the case $y\in\G\se\G_r$, it could be similarly proved by the proof of Lemma 3.1 of \cite{DLLY17}
that $G_r(x,y)=G_r(y,x)$ for $x\in\G_{b,a}$. For the case $y\in\G_r\se\G$, the function $G_r(\cdot,y)$ is
defined as $G_r(\cdot,y): =\Phi_{\kappa_1}(\cdot,y)+G_r^{1}(\cdot,y)$ in $\Om_{1,r}$ and
$G_r(\cdot,y):=G_r^{2}(\cdot,y)$ in $\Om_{2,r}$, where $(G^1_r(\cdot,y),G^2_r(\cdot,y))$ is the solution
to the scattering problem:
\be\label{x1}
\left\{\begin{array}{lllll}
\Delta G^1_r(\cdot,y)+\kappa_1^2 G^1_r(\cdot,y)=0&  {\rm in}\;\; \Om_{1,r}\\[2mm]
\Delta G^2_r(\cdot,y)+\kappa_2^2 G^2_r(\cdot,y)=0& {\rm in}\;\; \Om_{2,r}\\[2mm]
G^2_r(\cdot,y)-G^1_r(\cdot,y)=\Phi_{\kappa_1}(\cdot,y)&{\rm on}\;\;  \G_r\se\{y\}\\[2mm]
\ds\frac{\pa{G^2_r(\cdot,y)}}{\pa n}-\frac{\pa{G^1_r(\cdot,y)}}{\pa n}
=\frac{\pa{\Phi_{\kappa_1}(\cdot,y)}}{\pa n}&{\rm on}\;\;  \G_r\se\{y\}\\[2mm]
\ds\lim_{|x|\to \infty}\sqrt{|x|}\left(\frac{\pa G^{j}_r(\cdot,y)}{\pa |x|}
-{\rm i}\kappa_j G^j_r(\cdot,y)\right)=0& {\rm for}\;\; j=1,2.
\end{array}\right.
\en

To show the existence of a unique solution of Problem (\ref{x1}), define
$E^1_r(\cdot,y):=G^1_r(\cdot,y)$ and $E^2_r(\cdot,y):=G^2_r(\cdot,y)-\Phi_{\kappa_2}(\cdot,y)$
which satisfy Problem (\ref{x1}) with the boundary data replaced with
\ben\label{x11}
h_1(z):=\Phi_{\kappa_1}(z,y)-\Phi_{\kappa_2}(z,y)\quad{\rm and}\quad
h_2(z):=\frac{\pa\Phi_{\kappa_1}(z,y)}{\pa n(z)}-\frac{\partial\Phi_{\kappa_2}(z,y)}{\partial n(z)}.
 \enn
A direct calculation yields that $h_1(\cdot)\in C^1(\G_r)\cap L_{\rm loc}^2(\G_r)$ and
$h_2(\cdot)\in C(\G_r)\cap L_{\rm loc}^2(\G_r)$ for any $y\in \G_r\se\G$.
By a similar argument as in \cite{SZ99} it can be proved that Problem (\ref{x1}) is well-posed
in the above setting for $h_1$ and $h_2$, which implies that $E_r^1(\cdot,y)\in H^1_{\rm loc}(\Omega_{1,r})$
and $E_r^2(\cdot,y)\in H^1_{\rm loc}(\Omega_{2,r})$.

We now prove that $G_r(x,y)=G_r(y,x)$ for $y\in \G_r\se\G$ and $x\in\G_{b,a}$.
For $s=x,y$, let $B_{\epsilon}(s)$ denote a ball centered at $s$ with sufficiently small radius
$\epsilon>0$ such that $B_{\epsilon}(x)\cap B_{\epsilon}(y)=\emptyset$. Let $B^+_{\epsilon}(y):=B_{\epsilon}(y)\cap\Omega_{1,r}$, $B^-_{\epsilon}(y):=B_{\epsilon}(y)\cap\Omega_{2,r}$, and $B_{r'}:=\{x\in\R^2: |x|=r'>r\}$. Using the Green's theorem in $(\Omega_{1,r}\cap B_{r'})\se \ov{(B_{\epsilon}(x)\cup B^+_{\epsilon}(y))}$ gives
\ben
0&=&\int_{(\Omega_{1,r}\cap B_{r'})\se\ov{(B_{\epsilon}(x)\cup B^+_{\epsilon}(y))}}(\Delta G_r(z,x)G_r(z,y)-\Delta G_r(z,y)G_r(z,x)){\rm d}z\\
&=&\left\{\int_{\pa B^+_{r'}}+\int_{\pa B_{\epsilon}(x)}-\int_{\hat{\G}}+\int_{\pa B_{\epsilon}^+(y)\se \G_r}\right\}\left(\frac{\pa G_r(z,x)}{\pa n(z)}G_r(z,y)-\frac{\pa G_r(z,y)}{\pa n(z)}G_r(z,x)\right){\rm d}s(z)
\enn
where $\hat{\G}:=(B_{r'}\cap\G_r)\se B_{\epsilon}(y)$ and $\pa B_{r'}^{+}:=\{x\in\R^2: |x|=r', x_2>0 \}$. Using the Green's theorem again
in $(\Omega_{2,r}\cap B_{r'})\se B^-_{\epsilon}(y)$ yields
\ben
0&=&\int_{(\Omega_{2,r}\cap B_{r'})\se\ov{ B^-_{\epsilon}(y)}}(\Delta G_r(z,x)G_r(z,y)-\Delta G_r(z,y)G_r(z,x)){\rm d}z\\
&=&\left\{\int_{\pa B^-_{r'}}+\int_{\hat{\G}}+\int_{\pa B_{\epsilon}^-(y)\se \G_r}\right\}\left(\frac{\pa G_r(z,x)}{\pa n(z)}G_r(z,y)-\frac{\pa G_r(z,y)}{\pa n(z)}G_r(z,x)\right){\rm d}s(z)
\enn
where $\pa B_{r'}^{-}:=\{x\in\R^2: |x|=r', x_2<0 \}$. Then, combining the above two equalities leads to
\be\no
0&=&\left\{\int_{\pa B_{r'}}+\int_{\pa B_{\epsilon}(x)}+\int_{\pa B_{\epsilon}(y)}\right\}\left(\frac{\pa G_r(z,x)}{\pa n(z)}G_r(z,y)-\frac{\pa G_r(z,y)}{\pa n(z)}G_r(z,x)\right){\rm d}s(z)\\\label{x111}
&=:& I_1+I_2+I_3.
\en
It follows from the Sommerfeld radiation condition that $I_1\rightarrow 0$ as $r'\rightarrow \infty$
and from the mean value theorem that $I_2\rightarrow G_r(x,y)$ as $\epsilon\rightarrow 0$.
For the term $I_3$, by noting $G_r(z,y)=E(z,y)+\Phi(z,y)$ with $E(z,y)=E_r^{j}(z,y)$ and $\Phi(z,y)=\Phi_{\kappa_j}(z,y)$
for $z\in\Om_{j,r}$ $(j=1,2)$, we can split $I_3$ into two parts, that is, $I_3=I_{31}+I_{32}$, such that
\be\label{x222}
I_{31}:=\int_{\pa B_{\epsilon}(y)}\frac{\pa G_r(z,x)}{\pa n(z)}G_r(z,y){\rm d}s(z)\to 0\qquad{\rm as}\;\;\epsilon\to 0
\en
since $G_r(\cdot,y)\in H^{1,p}(B_{\epsilon}(y))$ for any fixed $p\in (1,2)$.
Further, we have
\be\no
I_{32}:&=&-\int_{\pa B_{\epsilon}(y)}\frac{\pa G_r(z,y)}{\pa n(z)}G_r(z,x){\rm d}s(z)\\\no
&=&-\int_{\pa B_{\epsilon}(y)}\frac{\pa E(z,y)}{\pa n(z)}G_r(z,x){\rm d}s(z)
       -\int_{\pa B_{\epsilon}(y)}\frac{\pa \Phi(z,y)}{\pa n(z)}G_r(z,x){\rm d}s(z)\\\label{x333}
&=&: I_{321}+I_{322}.
\en
By recalling $E_r^j(\cdot,y)\in H^1_{\rm loc}(\Omega_{j,r})$, one has $I_{321}\to 0$ and $I_{322}\to -G_r(y,x)$ as $\epsilon\to 0$, which combines with (\ref{x222}) and (\ref{x333}).
to show $I_3\to -G_r(y,x)$ as  $\epsilon\to 0$. Therefore, we have proved the reciprocity relation $G_r(x,y)=G_r(y,x)$ from equality (\ref{x111}) for all $x\in \G_{b,a}$ and $y\in\partial D$.

We next claim that $H^*$ is injective. To see this, let $H^*\psi=0$ on $\G_{b,a}$ for some $\psi\in H^{\frac{1}{2}}(\pa D)$ and define the function
\ben
\tilde{w}(x): = \int_{\pa D}G_r(x,y)\ov{\psi(y)}{\rm d}s(y)\qquad x\in\R^2.
\enn
Then, $\tilde{w}(x)=0$ on $\G_{b,a}$. The analyticity and the uniqueness of the Dirichlet boundary value problem in $U_b:=\{(x_1,x_2)\in\R^2:x_2>b\}$ implies $\tilde{w}=0$ in $U_b$.
Using the analytic continuation as well as the continuous conditions of $G_r$, $\pa_n G_r$ across $\G_r\cap\G$, we have $\tilde{w}=0$ in $\R^2\se\ov{D}$.
Since $G_r(x,y)$ has the same singularity with $\Phi_{\kappa_1}(x,y)$ at $x=y$ for $y\in\pa D$, we deduce that $\tilde{w}$ solves the homogeneous Dirichlet problem
\ben
\left\{\begin{array}{lllll}
\Delta \tilde{w}+\kappa_1^2\tilde{w}=0 &\quad {\rm in}\;\; D,\\[2mm]
\tilde{w}=0 &\quad {\rm on}\;\; \pa D.
\end{array}\right.
\enn
By the assumption on $\kappa^2_1>0$, we conclude $\tilde{w}=0$ in $D$, whence $\psi=0$
follows from the jump relationship of $\pa_n \tilde{w}$ on $\pa D$. Therefore, $H^*$ is injective
which means that ${\rm Range}(H)$ is dense in $H^{-\frac{1}{2}}(\pa D)$. The proof is thus complete.
\end{proof}
\begin{remark}\label{remark1} {\rm
The assumption in Lemma \ref{lem} could be removed by choosing a suitable $r>0$, which may be seen by
the strong monotonicity property and the continuous dependence of the Dirichlet eigenvalues of $-\Delta$
with respect to the domain \cite{L86}. We refer to Remark 3.2 in \cite{DLLY17} for details.
}
\end{remark}

By Lemma \ref{lem}, it is known that there exists a unique solution $(v(\cdot,z),w(\cdot,z))\in \ov{X}\times L^2(D)$ to Problem (\ref{d6})-(\ref{d7}) for each $z\in D$. With this result, we can establish the following two important properties for $F(I+\eta T)^{-1}$ which will play a key role in proposing our sampling method.

\begin{theorem}\label{thm3.1}
If $\kappa_1>0$ is not a transmission eigenvalue of Problem (\ref{d6})-(\ref{d7}), then $F(I+\eta T)^{-1}: \ov{X}\rightarrow L^2(\G_{b,a})$ is injective with dense range.
\end{theorem}

\begin{proof}
Let $F(I+\eta T)^{-1}v=0$ for some $v\in\ov{X}$ and set $w:=(I+\eta T)^{-1}v$.
Define the function
\ben
\xi(x):=-\eta \int_{D}G_r(x,z)w(z){\rm d}z\qquad{\rm for}\;\; x\in\R^2.
\enn
Then, $\xi=0$ on $\G_{b,a}$ follows from $Fw=0$. By the analyticity of $\xi$ in $\G_b:=\{(x_1,x_2)\in\R^2:x_2=b\}$, we conclude $\xi=0$ on $\G_b$. The uniqueness of the Dirichlet boundary value problem in $U_b:=\{(x_1,x_2)\in\R^2:x_2>b\}$ implies that $\xi=0$ in $U_b$. Then the analytic continuation and the continuity of $G_r$ and $\pa_n G_r$ across $\G_r\cap\G$ lead to $\xi=0$ in $\R^2\setminus \ov{D}$. Thus, we conclude that $(v,w)\in\ov{X}\times L^2(D)$  is a solution of Problem (\ref{d6})-(\ref{d7}) with the homogeneous conditions. Since
$\kappa_1>0$ is not a transmission eigenvalue, we derive $v=0$ which means that $F(I+\eta T)^{-1}$ is injective.

To show the denseness we assume that  there exists $g\in L^2(\G_{b,a})$ such that
\be\label{x444}
(F(I+\eta T)^{-1}v_h, g)=0\qquad {\rm for\;all\;}v_h\in X.
\en
Then by (\ref{d4}), we have
\ben
\int_{\G_{b,a}}\int_{\G_{b,a}}(u(x,y)-G_r(x,y))h(y){\rm d}s(y)\ov{g(x)}{\rm d}s(x)=0\qquad{\rm for\; all}\;h\in L^2(\G_{b,a}),
\enn
whence
\ben
\int_{\G_{b,a}}(u(y,x)-G_r(y,x))\ov{g(x)}{\rm d}s(x)=0\qquad{\rm for\; all}\;y\in \G_{b,a}.
\enn
A similar discussion implies
\ben
\int_{\G_{b,a}}(u(y,x)-G_r(y,x))\ov{g(x)}{\rm d}s(x)=0\qquad{\rm for\; all}\;y\in\R^2\setminus \ov{D}.
\enn
Thus, we conclude that $v_{\ov{g}}\in \ov{X}$ and $(\int_{\G_{b,a}}u(y,x)\ov{g(x)}{\rm d}s(x))|_D$,
are the solutions of the homogeneous problem (\ref{d6})-(\ref{d7}). Hence, $v_{\ov{g}}=0$ in $D$
due to the assumption on $\kappa_1$ which further implies $v_{\ov{g}}=0$ in $\Omega_{1,r}\setminus\G_{b,a}$
from the analytic continuation. Using the jump relations of the normal derivative of $v_{\ov{g}}$
gives that $g=0$. This means that, by (\ref{x444}) ${\rm Range}(F(I+\eta T)^{-1}(\ov{X}))$ is dense
in $L^2(\G_{b,a})$.
\end{proof}

\begin{theorem}\label{thm3.2}
If $\kappa_1>0$ is not a transmission eigenvalue Problem (\ref{d6})-(\ref{d7}), then $z\in D$ if and only
if $G_r(\cdot, z)|_{\Gamma_{b,a}}\in {\rm Range}(F(I+\eta T)^{-1})$.
\end{theorem}

\begin{proof}
Since $\kappa_1>0$ is not a transmission eigenvalue, there exists a unique solution
$(v(\cdot,z),w(\cdot,z))\in \ov{X}\times L^2(D)$ to Problem (\ref{d6})-(\ref{d7}) if $z\in D$.
Then we define a function $\wi{v}_z$ by
\be\label{s2}
\widetilde{v}_z(\cdot):=\left\{\begin{array}{l}
              G_r(\cdot,z) \qquad\qquad\quad\; \textrm{in}\;\; \Omega_1, \\[1mm]
              w(\cdot,z)-v(\cdot,z) \qquad \textrm{in}\;\; D,\\[1mm]
              G_r(\cdot,z) \qquad\qquad\quad\; \textrm{in}\;\; \Omega_{2,r}.
            \end{array}
\right.
\en
It is easily verified that $\widetilde{v}_z$ solves Problem (\ref{c2}) with
$\beta(\cdot):=\eta v(\cdot,z)$.
It follows from \cite{YLZ17} that $\widetilde{v}_z(\cdot)$ satisfies the Lippmann-Schwinger equation
\be\label{s1}
\widetilde{v}_z(x)+\eta\int_D G_r(x,y)\widetilde{v}_z(y){\rm d}y=-\eta\int_DG_r(x,y)v(y,z){\rm d}y
\en
for $x\in\R^2$, which may be reformulated in the form in $D$:
\ben
\widetilde{v}_z(\cdot) = -\eta T(v(\cdot,z)+\widetilde{v}_z(\cdot))\quad{\rm or}\quad
(\widetilde{v}_z(\cdot)+v(\cdot,z)) + \eta T(\widetilde{v}_z(\cdot)+v(\cdot,z)) = v(\cdot,z).
\enn
Then we have $\widetilde{v}_z(\cdot)+v(\cdot,z)=(I+\eta T)^{-1}v(\cdot,z)$, whence
$F(I+\eta T)^{-1}v(\cdot,z)=F(\wi{v}_z(\cdot)+v(\cdot,z))=\widetilde{v}_z(\cdot)|_{\Gamma_{b,a}}=G_r(\cdot,z)|_{\Gamma_{b,a}}$ follows from (\ref{s1}) and (\ref{s2}). Hence, it holds
$G_r(\cdot, z)|_{\Gamma_{b,a}}\in {\rm Range}(F(I+\eta T)^{-1})$ if $z\in D$.

Conversely, assume on the contrary $z\notin D$. Since $G_r(x,z)|_{\G_{b,a}}\in{\rm Range}(F(I+\eta T)^{-1})$,
there exists some $v\in\ov{X}$ such that $F(I+\eta T)^{-1}v=G_r(x,z)|_{\G_{b,a}}$. Set $w:=(I+\eta T)^{-1}v$
and define
\ben
\mu(x): = -\eta\int_{D}G_r(x,z)w(z){\rm d}z\qquad x\in\R^2.
\enn
It is checked that $\mu(x)$ is the solution to Problem (\ref{c2}) with the data $\eta v$. Hence,
we have $F(I+\eta T)^{-1}v=\mu|_{\Gamma_{b,a}}$ and then $\mu(x)=G_r(x,z)$ on $\Gamma_{b,a}$.
Using the analyticity and the uniqueness of the Dirichlet boundary value problem in $U_b$ again,
we conclude that $\mu(\cdot)=G_r(\cdot,z)$ on $U_b$.
We further have $\mu(x)=G_r(x,z)$ for $x\in \R^2\backslash\{D\cup\{z\}\}$ from the analytic continuation.
Note that $G_r(x,z)$ has a singularity at $x=z$ while $\mu(x)$ is smooth at $x=z$, which leads to a
contradiction. The proof is thus completed.
\end{proof}

With the aid of Theorems \ref{thm3.1} and \ref{thm3.2}, we are now able to formulate the main result
of this paper in the following theorem.

\begin{theorem}\label{thm3.3}
Assume that $\kappa_1>0$ is not a transmission eigenvalue of Problem (\ref{d6})-(\ref{d7}). Then we have

1) if $z\in D$, then for each $\epsilon>0$ there exists a solution $g_{z,\epsilon}\in L^2(\G_{b,a})$ of
the inequality
\be\label{d8}
\|{\bf N}_{\rm Mod}g_{z,\epsilon}(\cdot)-G_r(\cdot,z)\|_{L^2(\G_{b,a})}<\epsilon
\en
such that
\ben
\lim_{z\rightarrow \pa D\cap\G}\|g_{z,\epsilon}\|_{L^2(\G_{b,a})}=\infty \quad{\rm and}\quad
\lim_{z\rightarrow \pa D\cap\G}\|v_{g_{z,\epsilon}}\|_{L^2(D)}=\infty
\enn
where $v_{g_{z,\epsilon}}$ is defined by (\ref{d3}) with the density $g_{z,\epsilon}$;

2) if $z\in \R^2\setminus\ov{D}$, then for each $\epsilon>0$ and $\delta>0$ there exists a solution $g_{z,\epsilon}^{\delta}\in L^2(\G_{b,a})$ of the inequality
\be\label{d9}
\|{\bf N}_{\rm Mod}g_{z,\epsilon}^{\delta}(\cdot)-G_r(\cdot,z)\|_{L^2(\G_{b,a})}<\epsilon+\delta
\en
such that
\ben
\lim_{\delta\rightarrow 0}\|g^{\delta}_{z,\epsilon}\|_{L^2(\G_{b,a})}=\infty \quad{\rm and}\quad
\lim_{\delta\rightarrow 0}\|v_{g^{\delta}_{z,\epsilon}}\|_{L^2(D)}=\infty
\enn
where $v_{g^{\delta}_{z,\epsilon}}$ is defined by (\ref{d3}) with the density $g^{\delta}_{z,\epsilon}$.
\end{theorem}

\begin{proof}
If $z\in D$, let $(v(\cdot, z),w(\cdot,z))\in\ov{X}\times L^2(D)$ denote the unique solution of
Problem (\ref{d6})-(\ref{d7}). By Theorem \ref{thm3.2}, we have
$F(I+\eta T)^{-1}v(\cdot, z)=G_r(\cdot,z)|_{\G_{b,a}}$. Moreover, for each $\epsilon>0$ there exists
$g_{z,\epsilon}\in L^2(\G_{b,a})$ such that $\|v_{g_{z,\epsilon}}(\cdot)-v(\cdot,z)\|_{L^2(D)}<\epsilon$.
This, combined with the equality $F(I+\eta T)^{-1}={\bf N}_{\rm Mod}$, implies that
\ben
\|{\bf N}_{\rm Mod}g_{z,\epsilon}(\cdot)-G_r(\cdot,z)\|_{L^2(\G_{b,a})}<C_0\epsilon
\enn
for a fixed constant $C_0>0$.

Assuming that $\|g_{z,\epsilon}\|_{L^2(\G_{b,a})}\leq C$ uniformly as $z\rightarrow \pa D\cap\G$,
we have by (\ref{d3}) that $||v_{g_{z,\epsilon}}||_{L^2(D)}\leq C$ as $z\rightarrow \pa D\cap\G$.
Since $\|v_{g_{z,\epsilon}}(\cdot)-v(\cdot, z)\|_{L^2(D)}<\epsilon$, we conclude that
$\|v(\cdot,z)\|_{L^2(D)}\leq C+C_0\epsilon$. Thus it follows from the equality
$w(\cdot, z)=(I+\eta T)^{-1}v(\cdot, z)$ that $\|w(\cdot,z)\|_{L^2(D)}\leq C_2$ for a fixed constant $C_2>0$
as $z\rightarrow\pa D\cap\G$. Using the boundedness of the operator $T:L^2(D)\to H^1(D)$ as well as
the trace theorem gives that $\|v(\cdot,z)-w(\cdot,z)\|_{H^{1/2}(\pa D)}\leq C_3$ for a fixed constant $C_3>0$.
However, this is a contradiction since, by the boundary condition $v(\cdot,z)-w(\cdot,z)=G_r(\cdot,z)$ on $\pa D$
we have $\|G_r(\cdot,z)\|_{H^{1/2}(\pa D)}\rightarrow\infty$ as $z\rightarrow\pa D\cap\G$.

If $z\in\R^2\setminus\ov{D}$, it is known by theorem \ref{thm3.2} that $G_r(\cdot,z)|_{\G_{b,a}}\notin {\rm Range}(F(I+\eta T)^{-1})$. Moreover, it follows from theorem \ref{thm3.1} that $F(I+\eta T)^{-1}$ is injective with dense range in $L^2(\G_{b,a})$. So, for the operator equation $F(I+\eta T)^{-1}v(\cdot)=G_r(\cdot,z)|_{\G_{b,a}}$, there always exists a regularized solution $v_z^{\alpha}\in\ov{X}$ to its regularized equation
\ben
\alpha v_z^{\alpha}+A^*Av_z^{\alpha}=A^*(G_r(\cdot,z)|_{\G_{b,a}})
\enn
which can be represented as
\ben
v_z^{\alpha}=\sum_{n=1}^{\infty}\frac{\mu_n}{\alpha+\mu_n^2}(G_r(\cdot,z)|_{\G_{b,a}}, g_n)\varphi_n
\enn
with $(\mu_n, \varphi_n, g_n)$ a singular system for $A$,
where $A:=F(I+\eta T)^{-1}$ and $\alpha>0$ is a regularization parameter. By theorem 2.13 of  \cite{CC06} it is also known that $v_z^{\alpha}$ is the minimizer of the Tikhonov functional.
Therefore, for every $\delta>0$, we can deduce by choosing $\alpha>0$ that
\be\label{d10}
\|F(I+\eta T)^{-1}v_z^{\alpha}-G_r(\cdot,z)\|_{L^2(\G_{b,a})}<\delta.
\en
Since $G_r(\cdot, z)|_{\G_{b,a}}\notin {\rm Range}(F(I+\eta T)^{-1})$, using the Picard's theorem implies that $\|v_z^{\alpha}\|_{L^2(D)}\rightarrow \infty$ as $\alpha\rightarrow 0$. By noticing  $v_z^{\alpha}\in\ov{X}$, for sufficiently small $\epsilon>0$ there then exists $g_{z,\epsilon}^{\delta}\in L^2(\G_{b,a})$ such that
\be\label{d11}
\|v_{g_{z.\epsilon}^{\delta}}-v_z^{\alpha}\|_{L^2(D)}<\epsilon
\en
which means
\be\label{d12}
\|F(I+\eta T)^{-1}v_{g_{z.\epsilon}^{\delta}}-F(I+\eta T)^{-1}v_z^{\alpha}\|_{L^2(\G_{b,a})}<\epsilon
\en
in the sense of omitting a constant. We now combine (\ref{d10}) and (\ref{d12}) to obtain
\ben
&&\|{\bf N}_{\rm Mod}g_{z,\epsilon}^{\delta}(\cdot)-G_r(\cdot,z)\|_{L^2(\G_{b,a})}\\
&=&\|F(I+\eta T)^{-1}v_{g_{z.\epsilon}^{\delta}}(\cdot)-G_r(\cdot,z)\|_{L^2(\G_{b,a})}\\
&\leq& \|F(I+\eta T)^{-1}v_{g_{z.\epsilon}^{\delta}}-F(I+\eta T)^{-1}v_z^{\alpha}\|_{L^2(\G_{b,a})}
+\|F(I+\eta T)^{-1}v_z^{\alpha}-G_r(\cdot,z)\|_{L^2(\G_{b,a})}\\
&<&\epsilon+\delta.
\enn
Finally, the inequality (\ref{d11}) and the fact $||v_z^{\alpha}||_{L^2(D)}\rightarrow \infty$ as $\alpha\rightarrow 0$ gives
\ben
\lim_{\alpha\rightarrow 0}\|v_{g^{\delta}_{z,\epsilon}}\|_{L^2(D)}=\infty
\enn
whence
\ben
\lim_{\alpha\rightarrow 0}\|g^{\delta}_{z,\epsilon}\|_{L^2(\G_{b,a})}=\infty.
\enn
This ends the proof due to $\alpha\rightarrow 0$ as $\delta\rightarrow 0$.
\end{proof}

By Theorem \ref{thm3.3}, it is known that the solution of the modified near-field equation (\ref{d2}) in the sense of inequalities (\ref{d8}) and (\ref{d9}) has totally different behaviors when the sampling point $z$ lies inside or outside of $D$, which gives a qualitative way to visualize the support $D$. Based on this observation, we define the indicator function
\ben
{\rm Ind}(z):=1/\|g_z\|_{L^2(\G_{b,a})}
\enn
by the solution of (\ref{d8}) and (\ref{d9}). It is easily seen that the indicator function ${\rm Ind}(z)$ is small when the sampling point $z$ approaches the local perturbation of the interface $\G$ from inside of $D$, which can provide a fast imaging algorithm. The following procedure shows how to recover the interface $\Gamma$ by the indicator function ${\rm Ind}(z)$.
\begin{algorithm}\caption{Reconstruction of locally rough interfaces by linear sampling method}\label{alg1}
\begin{itemize}
\item Select a rectangular grid $S$ containing the local perturbation of the scattering interface $\G$.
\item Choose $r>0$ to be large enough such that $\Omega_1\subset\Omega_{1,r}$, that is, the local perturbation
of the interface $\G$ lies totally in $\Omega_{1,r}$. Then solve the  scattering problem (\ref{c1})
and (\ref{a1})-(\ref{a2}) for each $y\in\G_{b,a}$ to obtain the scattered field data $G_r^s(x,y)$
and $u^s(x,y)$ by the Nystr\"{o}m method (cf.\cite{LYZ13}).
\item For each sampling point $z\in S$, solve the modified near-field equation (\ref{d2}) by the
regularization method to obtain the solution $g_z$ and compute the value of the indicator function ${\rm Ind}(z)$.
\item Choose a cut-off value $C>0$ and so that it is numerically reliable that $z\in D$ if and
only if ${\rm Ind}(z)\leq C$.
\end{itemize}
\end{algorithm}

\section{Numerical experiments}\label{sec4}
\setcounter{equation}{0}

In this section, we carry out several numerical examples to illustrate the effectiveness of the modified LSM proposed in Theorem \ref{thm3.3}. It is known by the analyticity of the kernel
$u^{\rm s}(\cdot,y)-G^{\rm s}(\cdot,y)$ that the operator ${\bf N}_{\rm Mod}$ is compact on $L^2(\G_{b,a})$. Hence, equation (\ref{d2}) is severely ill-posed and is solved by a regularization method with the regularization parameter $\alpha(z)$ chosen by the Morozov's discrepancy principle. However, a number of numerical examples we carried out show that the regularization parameter $\alpha(z)$ can be taken as a fixed parameter. In this paper, we choose
$\alpha(z)=5\times 10^{-5}$. In addition, we take the wavenumber $\kappa(\cdot)$ to be
$\kappa_1=1$ and $\kappa_2=10$, where the wavelength is given as $\lambda=2\pi/\kappa_1=2\pi$.

In numerical experiments, the synthetic scattering data $u^{\rm s}(x,y)$ and $G_r^{\rm s}(x,y)$ are obtained by solving the scattering problem (\ref{a1})-(\ref{a2}) and (\ref{c1}) by the Nystr\"{o}m method \cite{LYZ13}. Thus, we can discretize the modified near-field operator ${\bf N}_{\rm Mod}$ into the following finite dimensional matrix
\ben
({\bf N}_{\rm Mod})_{N\times N}=(u^{\rm s}(x_p, y_q)-G_r^{\rm s}(x_p, y_q))_{1\leq p, q\leq N}
\enn
where $x_p$ is the measuring points equally distributed at $\G_{b,a}$ with $p=1,2,...,N$, and $y_q$ is the incident point sources which is also equally distributed at $\G_{b,a}$ with $q=1,2,...,N$.
Then we discrete the test function $G_r(x,z)$ as $G_r(x_p,z)$ with $1\leq p\leq N$ with the sampling points $z$ belonging to a rectangular grid $S$ containing the local perturbation of $\G$.
Therefore, equation (\ref{d2}) can discretized into the following regularized equation
\ben
U_{\rm Mod}(\alpha I+S_{\rm Mod})V_{\rm Mod}^*\hat{g}_z=({\bf N}_{\rm Mod})^*_{N\times N}(G_r(\cdot,z)|_{\G_{b,a}})
\enn
where $U_{\rm Mod}$, $S_{\rm Mod}$ and $V_{\rm Mod}$ are the related matrices in the SVD algorithm satisfying $U_{\rm Mod}S_{\rm Mod}V_{\rm Mod}^*=({\bf N}_{\rm Mod})_{N\times N}^*({\bf N}_{\rm Mod})_{N\times N}$. Then we can define the following discretized form
\ben
{\rm Ind}_N(z) = 1{\left/\left[\sum_{j=1}^{N}|\hat{g}_{z,j}|^2\right]^{\frac{1}{2}}\right.}
\enn
of the indicator function ${\rm Ind}(z)$, where $\hat{g}_z: =(\hat{g}_{z,1},\cdots,\hat{g}_{z,N})^T\in \C^N$.

If ${\rm Ind}_N(z)$ approximates ${\rm Ind}(z)$, then ${\rm Ind}_N(z)$ should be large in $D$ and very small in $\Omega_1$. To present the results under the same standard, ${\rm Ind}_N(z)$ is normalized to a new indicator function
\ben
{\rm NInd(z)}:={\rm Ind}_N(z)/\max_{z\in S}{\rm Ind}_N(z)
\enn
which will be used to reconstruct the local perturbation of the scattering interface $\G$.

To examine the stability of the sampling method, we also exam our method with noisy data. To this end,
let $(\zeta)_{N\times N}=(\zeta_1)_{N\times N}+{\rm i}(\zeta_2)_{N\times N}$ stand for a complex-valued
matrix with its real part $\zeta_1$ and imaginary part $\zeta_2$ consisting of random numbers which obey
the normal distribution $N(0,1)$. Thus, the scattering data with noisy is given as
\ben
(({\bf N}_{\rm Mod})_{N\times N})_{\delta}:=({\bf N}_{\rm Mod})_{N\times N}
+\delta\frac{\zeta}{\|\zeta\|_2}\|({\bf N}_{\rm Mod})_{N\times N}\|_2
\enn
for relative error $\delta>0$.

As shown in Theorem  \ref{thm3.3}, we should remark that the modified LSM always works for every sufficiently
large $r>M$ as long as $\kappa_1>0$ is not a transmission eigenvalue of  Problem (\ref{d6})-(\ref{d7})
in $D=\Omega_2\cap \Omega_{1,r}$. Moreover, it is also noticed that  the computation of $G_r^{\rm s}(\cdot,z)$
may be time-consuming for sufficiently large $r$, which seems to infect the effectiveness of the modified LSM.
However, a variety of numerical experiments we have carried out imply that $G_r^{\rm s}(\cdot,z)$ decays to $0$
in any fixed bounded domain as $r\rightarrow \infty$. We illustrate this fact in the following table for $r=10^t$
with $t=2,4,6,8,10$, where the location of the incident point source is given at $z=(0,1)$, and the receivers
are given at $x_1=(-2,1)$, $x_2=(0,1)$ and $x_3=(2,1)$.

\begin{table}[tbhp]
\caption{\label{us_num} Numerical solutions of $G_r^{\rm s}(x,z)$
as the radius $r\to\infty$}
\centering
\begin{tabular}{ l     l     l}
\hline
 &$r$        &\qquad\;\;$G_r^{\rm s}(x_1,z)$  \qquad\qquad\qquad\;\;\;\; $G_r^s(x_2,z)$ \qquad\qquad\qquad\;\;\;\;\;\;$G^s(x_3,z)$ \\\hline
 &$10^2$  &\quad\;\;-0.0223-0.0045{\rm i}\qquad\qquad\quad -0.0241-0.0051{\rm i}\qquad\qquad\quad -0.0223-0.0045{\rm i}       \\
 &$10^4$     &1.0e-03$\cdot$(0.8017+2.8967{\rm i})   \quad1.0e-03$\cdot$(0.8008+2.8992{\rm i}) \quad\;1.0e-03$\cdot$(0.8017+2.8967{\rm i})   \\
 &$10^6$     &1.0e-03$\cdot$(-0.0312+0.1383{\rm i})\quad1.0e-03$\cdot$(-0.0312+0.1383{\rm i})\quad\
 1.0e-03$\cdot$(-0.0312+0.1383{\rm i})                  \\
 &$10^8$     &1.0e-05$\cdot$(-1.6947-0.8085{\rm i})  \quad1.0e-05$\cdot$(-1.6947-0.8085{\rm i})\quad\;\;
 1.0e-05$\cdot$(-1.6947-0.8085{\rm i})                \\
 &$10^{10}$ &1.0e-06$\cdot$(0.8272+1.7961{\rm i}) \quad1.0e-06$\cdot$(0.8272+1.7961{\rm i})\quad\;
 1.0e-06$\cdot$(0.8272+1.7961{\rm i})         \\
\hline
\end{tabular}
\end{table}

From Table \ref{us_num}, it could be deduced that for any given positive number $m>0$ we always can choose
sufficiently large $r=r(m)$ such that the estimate $\|{\bf N}-{\bf N}_{\rm Mod}\|_{L^2(\G_{b,a})}<10^{-m}$.
For example, $m>0$ is chosen to be $100$. In this case,
equation (\ref{d2}) could be reduced to solve its approximate equation
\be\label{e2}
({\bf N}g)(x)=\Phi_{\kappa_1}(x,z)\qquad{\rm for}\;\; x\in\G_{b,a}.
\en
Similarly, we have the following discrete regularized equation for (\ref{e2}):
\be\label{e4}
U(\alpha I+S)V^*\hat{g}_z=({\bf N})^*_{N\times N}(\Phi_{\kappa_1}(\cdot,z)|_{\G_{b,a}})
\en
where $U$, $S$ and $V$ are the related matrices in the SVD algorithm satisfying
$USV^*={\bf N}_{N\times N}^*({\bf N})_{N\times N}$.
Then Algorithm \ref{alg1} can be reformulated as follows.
\begin{algorithm}\caption{Reconstruction of locally rough interfaces by linear sampling method}\label{alg2}
\begin{itemize}
\item Select a rectangular grid $S$ containing the local perturbation of the scattering interface $\G$.
\item Solve the scattering problem (\ref{a1})-(\ref{a2}) for each $y\in\G_{b,a}$ by the Nystr\"{o}m
method (cf.\cite{LYZ13}).
\item Solve the discretization regularized equation (\ref{e4}) with the regularization parameter
$\alpha(z)=5\times 10^{-5}$ to obtain its solution $g_z^{\alpha}$.
\item Compute the normalized indicator function
\ben\label{e3}
{\rm NInd}(z):={\rm Ind}_N(z)/\max\limits_{z\in S} {\rm Ind}_N(z)
\enn
and then plot the mapping ${\rm NInd}(z)$ against $z$.
\end{itemize}
\end{algorithm}
\newline
\newline
\newline

\textbf{Example 1}. We first consider the interface $\G$ described by the function
\ben
f_1(x_1)=0.6\cdot\Omega_3(x_1)\qquad x_1\in\R^1,
\enn
where $\Omega_3(\cdot)$ denotes cubic B-spline function given by
\ben\no
\Omega_3(x_1)=\left\{\begin{array}{cc}
\ds\qquad\quad\;\;\;\;|x_1|^3/2-x_1^2+2/3& |x_1|\leq1,\\
\ds-|x_1|^3/6+x_1^2-2|x_1|+4/3 &\qquad 1<|x_1|<2,\\
\qquad\qquad\qquad\qquad\qquad\; 0 &|x_1|\geq2.
\end{array}\right.
\enn
which is twice continuously differentiable and has compactly support.

In this example, the local displacement lies totally above the plane $\G_0$. According to the location of the local displacement, we put the sampling points $z$ in the rectangular grid $(-10,10)\times (-1, 0.5)$ with step size $0.5$ in $x$-axis and $0.1$ in $y$-axis. The measurement was taken at $\G_{b,a}$ with $a=15$ and $b=0.55$, and measurement points are chosen as $N=601$ which are uniformly distributed over $\G_{b,a}$. As shown in Figure \ref{f2}, our inversion algorithm can present a satisfactory reconstruction in this case with exact data, $2\%$ noise and $5\%$ noise.

\begin{figure}
\centering
\subfigure[0\%]{\includegraphics[width=0.31\textwidth]{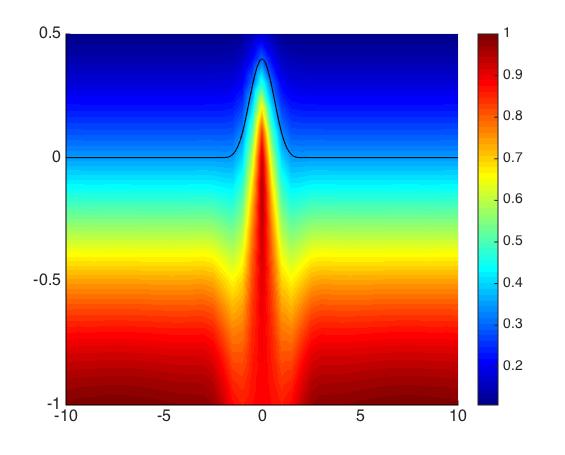}}
\subfigure[2\%]{\includegraphics[width=0.31\textwidth]{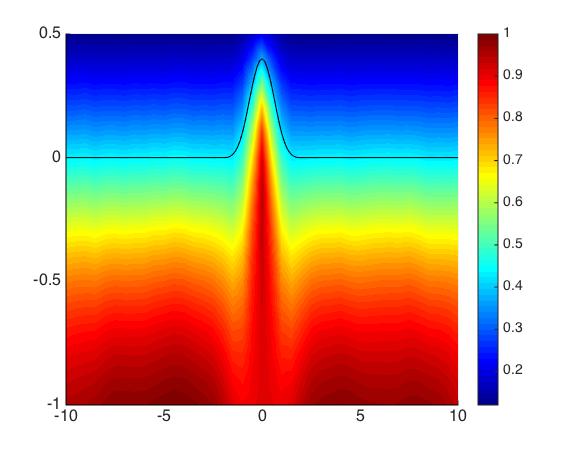}}
\subfigure[5\%]{\includegraphics[width=0.31\textwidth]{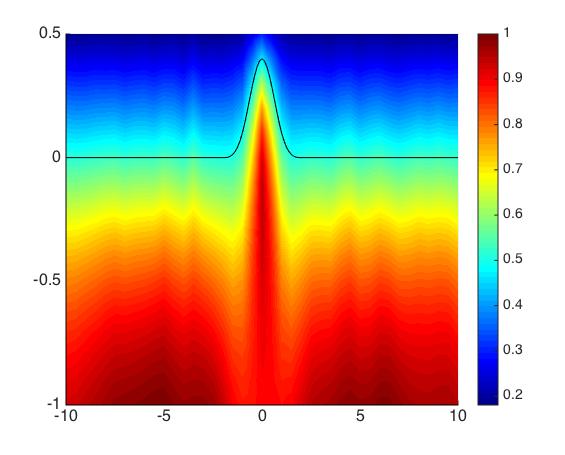}}
\caption{Reconstructions of the locally rough interface given in Example 1 from data with no noise (a), 2\% noise (b) and 5\% noise (c)}
\label{f2} 
\end{figure}

\textbf{Example 2}. In this example, the scattering interface $\G$ is described as
\ben
f_2(x_1)=\left[-0.3\exp(-x_1^2)-0.4\exp(-4(x_1-2)^2)-0.2\exp(-3(x_1+2)^2)\right]\cdot\chi(x_1),
\enn
where $\chi(x_1)\in C_0^{\infty}(\R^1)$ is defined by
 \begin{equation}\no
  \chi(x_1)=\left\{\begin{array}{cc}\qquad\qquad\qquad\qquad\;\; 1 &|x_1|\leq 4,\\\left[1+\exp\left(\frac{1}{5-|x_1|}+\frac{1}{4-|x_1|}\right)\right]^{-1} &\qquad
  4<|x_1|<5,\\\qquad\qquad\qquad\qquad\;\; 0 &|x_1|\geq 5.\end{array}\right.
\end{equation}
It is seen that the local perturbation of $f_2$ lies totally below the plane $\G_0$. In this example, we set the rectangular grid for $(-10, 10)\times (-1, 0.2)$ with step size $0.5$ in $x$-axis and $0.1$ in $y$-axis. The scattering data are measured at $N=601$ points which are uniformly distributed over $\G_{b,a}$ with $b=0.25$ and $a=15$. The reconstructions with no noise, $2\%$ noise and  $5\%$ noise were presented in Figure \ref{f3} which shows that our inversion algorithm can also provide a satisfactory reconstruction in this case.
\begin{figure}
\centering
\subfigure[0\%]{\includegraphics[width=0.31\textwidth]{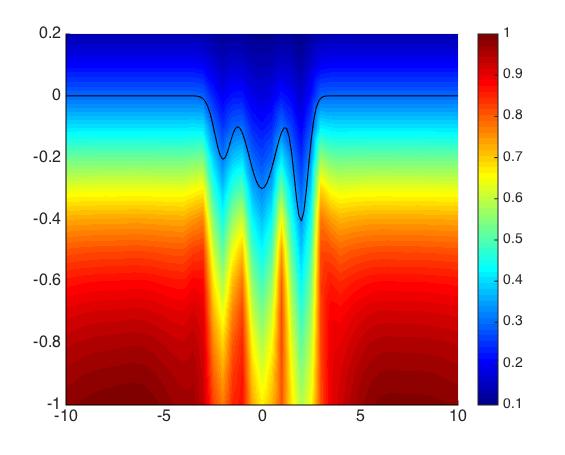}}
\subfigure[2\%]{\includegraphics[width=0.31\textwidth]{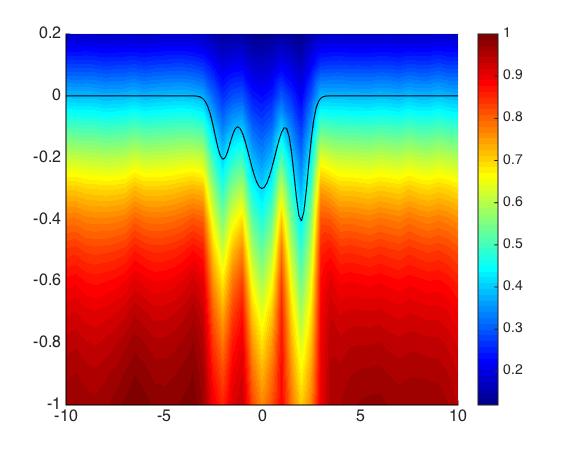}}
\subfigure[5\%]{\includegraphics[width=0.31\textwidth]{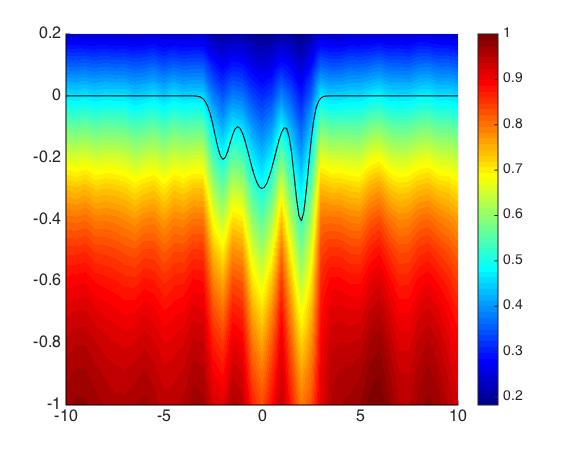}}
\caption{Reconstructions of the locally rough interface given in Example 2 from data with no noise (a),\;2\% noise (b) and 5\% noise (c)}
\label{f3} 
\end{figure}

\textbf{Example 3}. In this example, the scattering interface $\G$ is described by
a complex function defined as
\begin{equation}\no
  f_3(x_1)=\left\{\begin{array}{c}
                           \exp(16/(x_1^2-16)) \sin(\pi x_1)\qquad\;\; |x_1|<4, \\[1mm]
                           \qquad\qquad\qquad\qquad\qquad\;\;\;\; 0\;\;\qquad|x_1|\geq 4.
                          \end{array}\right.
\end{equation}
The above definition shows that part of $\G$ lies above $\G_0$ and other part lies below $\G_0$.
We choose the same rectangular grid and the same measurement points as in example 1.
 As seen in Figure \ref{f4}, the inversion algorithm remains valid to give a satisfactory reconstruction for a general locally rough interface, even at the noise level of $5\%$.

\begin{figure}
\centering
\subfigure[0\%]{\includegraphics[width=0.31\textwidth]{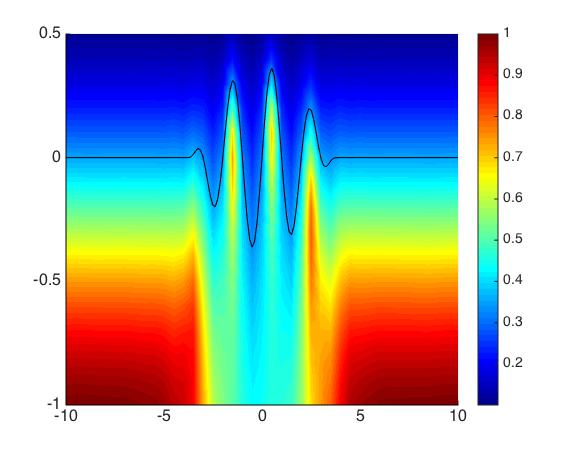}}
\subfigure[2\%]{\includegraphics[width=0.31\textwidth]{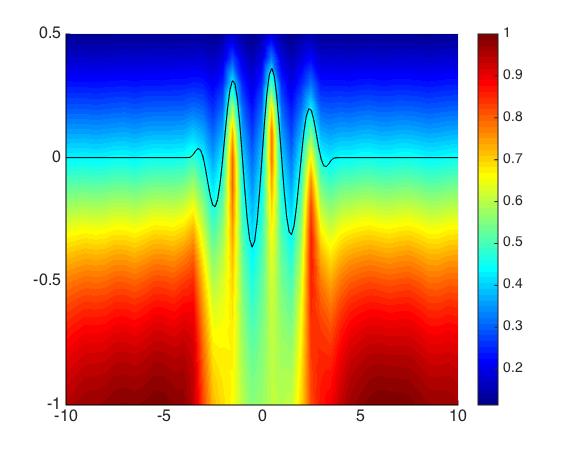}}
\subfigure[5\%]{\includegraphics[width=0.31\textwidth]{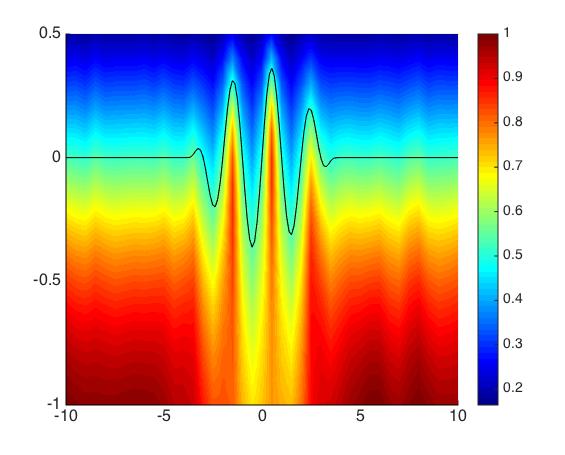}}
\caption{Reconstructions of the locally rough interface given in Example 3 from data with no noise (a), 2\% noise (b) and 5\% noise (c)}
\label{f4} 
\end{figure}

Notice the inversion algorithm depends on the number $N$ of measurement points, the measurement hight $b$ and the measurement width $a$. In the following experiments, we will focus on the influence of these three parameters on the numerical performance of the inversion algorithm. To show the influence of these parameters more clearly, we work only for the case with $2\%$ noise.

\textbf{Example 4}. In this example, we choose the same interface and the same sampling points with Example 1, and exam the influence of $N$. To this end, we fix $a=15$ and $b=0.55$, and present numerical results with $N=201$, $N=401$ and $N=601$ from left to right. As shown in Figure \ref{f5}, a better reconstruction is obtained as measurement points becomes larger.

\begin{figure}
\centering
\subfigure[N=201]{\includegraphics[width=0.31\textwidth]{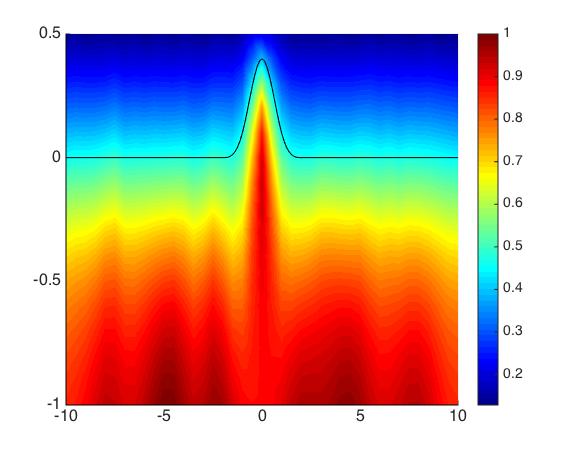}}
\subfigure[N=401]{\includegraphics[width=0.31\textwidth]{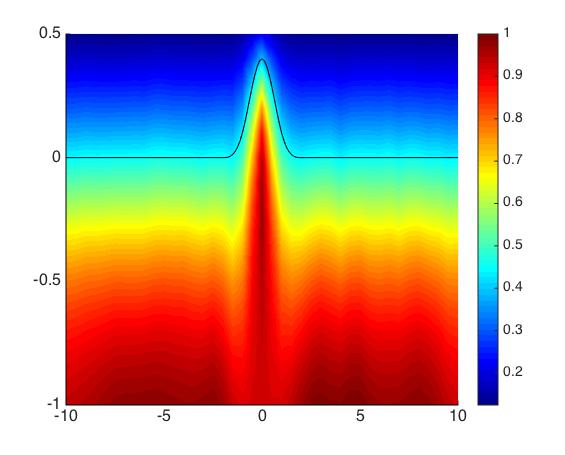}}
\subfigure[N=601]{\includegraphics[width=0.31\textwidth]{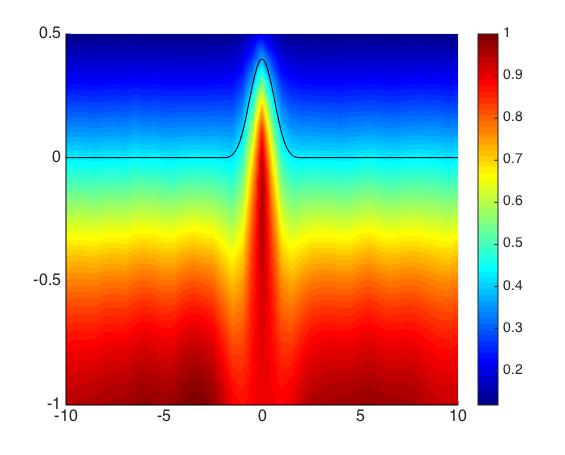}}
\caption{Reconstructions of the locally rough interface given in Example 4 from data with $N=201$ (a), $N=401$ (b) and $N=601$ (c) at the same level of 2\% noise}
\label{f5} 
\end{figure}

\textbf{Example 5}. In this example, we focus on the influence of the measurement hight $b$.
We choose the same scattering interface and rectangular grid as in example 2. Then we fix $a=15$ and $N=601$, and set the measurement hight $b=0.25$, $b=0.65$, and $b=1.05$, respectively. We present the reconstruction results in Figure \ref{f6} which shows that the quality of the reconstruction becomes worse when $b$ is large.

\begin{figure}
\centering
\subfigure[b=0.25]{\includegraphics[width=0.31\textwidth]{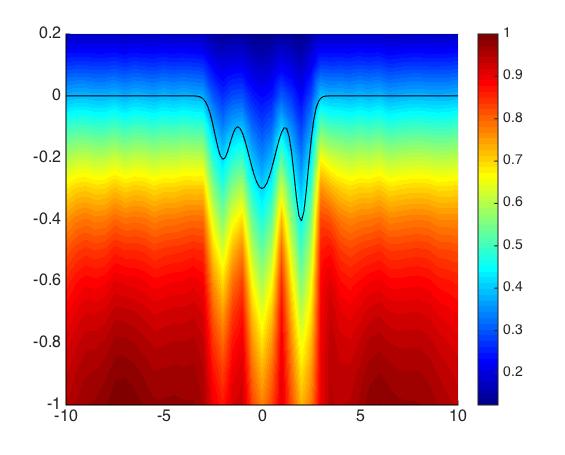}}
\subfigure[b=0.65]{\includegraphics[width=0.31\textwidth]{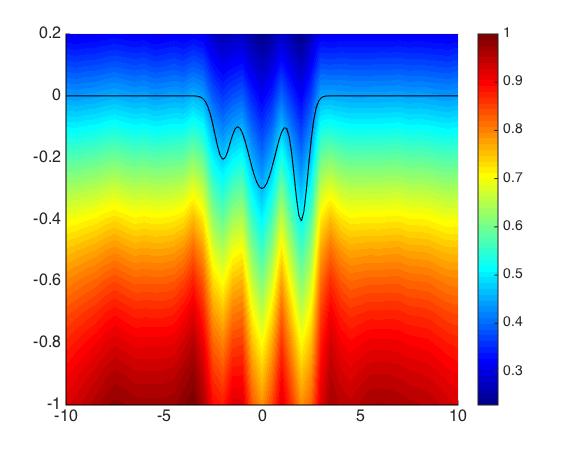}}
\subfigure[b=1.05]{\includegraphics[width=0.31\textwidth]{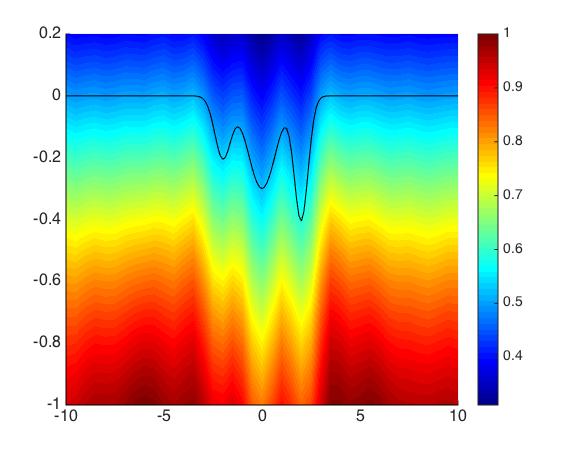}}
\caption{Reconstructions of the locally rough interface given in Example 5 from data with $b=0.25$ (a), $b=0.65$ (b) and $b=1.05$ (c) at the same level of 2\% noise}
\label{f6} 
\end{figure}

\textbf{Example 6}. In this example, we test the influence of the measurement width $a$. The scattering interface is given by $f(x_1)=0.6\cdot\Omega_3(x_1)-0.4\cdot\Omega_3(x_1-5)$.
We set the same rectangular grid $S$ as in Example 1 and fix $b=0.55$ and $N=601$. Figure \ref{f7} shows the reconstructions with $a=2$, $a=8$, and $a=14$, respectively, from left to right. As shown in Figure \ref{f7}, we can recover the essential feature of the locally rough interface  very well below the measurement curve $\G_{b,a}$.

\begin{figure}
\centering
\subfigure[a=2]{\includegraphics[width=0.31\textwidth]{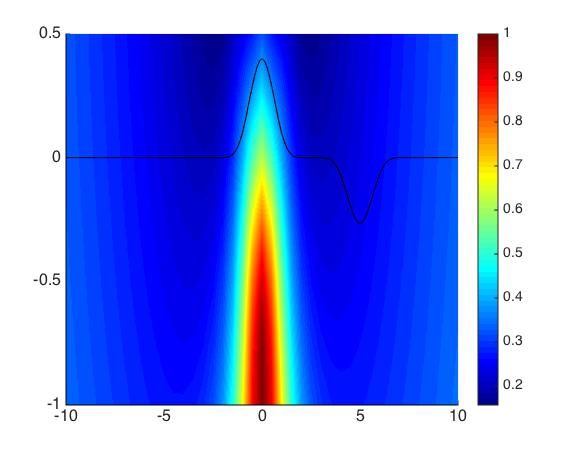}}
\subfigure[a=8]{\includegraphics[width=0.31\textwidth]{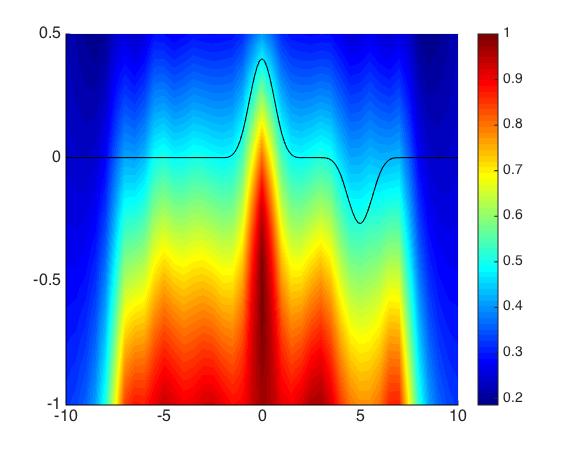}}
\subfigure[a=14]{\includegraphics[width=0.31\textwidth]{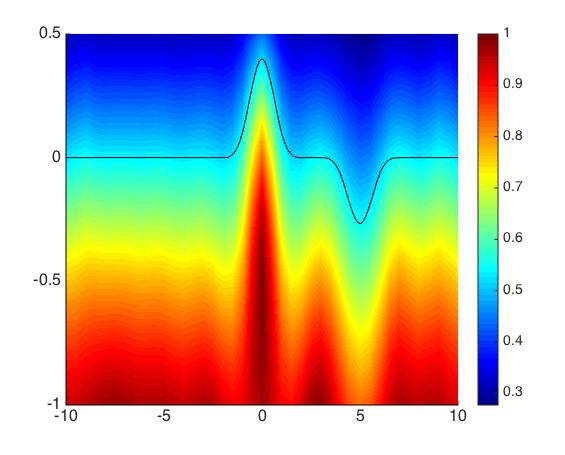}}
\caption{Reconstructions of the locally rough interface given in Example 5 from data with $a=2$ (a), $a=8$ (b) and $a=14$ (c) at the same level of 2\% noise}
\label{f7} 
\end{figure}

From the results shown in Figures \ref{f2}-\ref{f4}, it is easily verified that our proposed method in Theorem \ref{thm3.3} can recover the local perturbation of the scattering interface very well at  different noise levels. In addition, it is easily seen that our method could give a better reconstruction as $a$ and $b$ become larger which corresponds to more measurement data to be available. It is also noticed that the quality of the reconstructions become worse when $b$ becomes larger. In this case the scattered field decays when the measurement hight $b$ becomes large which leads to that lesser measurement data are available in this case.

\section{Conclusion}

In this paper, we proposed an extended sampling method to recover the shape and location of
a locally rough interface by near-field measurements above the interface.
The idea is mainly based on constructing a modified near-field operator equation by reducing
the interface scattering problem into an equivalently Lippmann-Schwinger integral equation in a bounded domain,
which generalized our previous work for a locally rough surface with a Dirichlet boundary condition.
Numerical experiments showed that the proposed method can provide a stable and satisfactory reconstruction
for locally rough interfaces with different image features. However, it is not clear how to
extend our method with mathematical justification to recover a non-locally rough interface.
Moreover, it is also challenging to develop a valid sampling method to simultaneously recover an
unbounded rough interface and bounded obstacles embedded in a lower half-space, which is more interesting
in practical applications. We hope to report these works in the future.

\section*{Acknowledgements}

This work was partly supported by the NNSF of China grants 11771349, 91730306 and 11601042.

\end{document}